\documentclass[12pt]{article} 
\usepackage[utf8]{inputenc}
\usepackage{index}
\usepackage[reqno]{amsmath}
\usepackage{fixmath}
\usepackage{amssymb, amsthm, enumerate}
\usepackage{mathtools}
\usepackage{fullpage}
\usepackage{hyperref}
\usepackage{xcolor}
\hypersetup{
    colorlinks,
    linkcolor={red!50!black},
    citecolor={blue!50!black},
    urlcolor={blue!80!black}
}
\usepackage{bbold}
\usepackage{authblk}

\title{Graph covers with two new eigenvalues}

\author[1]{Chris Godsil\thanks{C. Godsil acknowledges support of NSERC (Canada), Grant No. RGPIN-9439.}}
\author[1]{Maxwell Levit\thanks{M. Levit and O. Silina acknowledge the support of Chris Godsil.}}
\author[1]{Olha Silina$^\dagger$}
\affil[1]{Department of Combinatorics \& Optimization, University of Waterloo}

\date{October 7, 2020}

\DeclareMathOperator{\sym}{Sym}

\newcommand{\bb}[1]{\mathbb{#1}}
\newcommand{\mc}[1]{\mathcal{#1}}
\newcommand{\pa}[1]{\left(#1\right)} 
\newcommand{\IGNORE}[1]{}
\newcommand{\mat}[1]{\begin{matrix}#1\end{matrix}} 
\newcommand{\pmat}[1]{\pa{\mat{#1}}} 
\newtheoremstyle{plainsl}%
	{\topsep}
	{\topsep}
	{\slshape} 
	{}
	{\normalfont\bfseries}
	{.}
	{ }
	{}

\swapnumbers

{\theoremstyle{plainsl}
\newtheorem{theorem}{Theorem}[section]
\newtheorem{lemma}[theorem]{Lemma}
\newtheorem{prop}[theorem]{Proposition}
\newtheorem{corollary}[theorem]{Corollary}}
{\theoremstyle{remark}
}

\newcommand\sqr[2]{{\vbox{\hrule height.#2pt
    \hbox{\vrule width.#2pt height#1pt \kern#1pt
        \vrule width.#2pt}\hrule height.#2pt}}}

\renewcommand\qed{%
	\ifmmode\eqno\sqr53
	\else\nolinebreak\ \hfill\sqr53\medbreak\fi}

\begin{document}
 \bibliographystyle{acm}

\maketitle

\begin{abstract}
A certain signed adjacency matrix of the hypercube, which Hao Huang used last year to resolve the Sensitivity Conjecture, is closely related to the unique, 4-cycle free, 2-fold cover of the hypercube. We develop a framework in which this connection is a natural first example of the relationship between group valued adjacency matrices with few eigenvalues, and combinatorially interesting covering graphs. In particular, we define a \textbf{two-eigenvalue cover}, to be a covering graph whose adjacency spectra differs (as a multiset) from that of the graph it covers by exactly two eigenvalues. We show that walk regularity of a graph implies walk regularity of any abelian two-eigenvalue cover. We also give a spectral characterization for when a cyclic two-eigenvalue cover of a strongly-regular graph is distance-regular.

\end{abstract}

\section{Introduction}
In 1985, Arjeh Cohen and Jacques Tits proved the existence and uniqueness of $\widetilde{Q}_n$, a girth six, 2-fold cover of the hypercube $Q_n$. A 2-fold cover of a graph can be specified by a $\pm1$ signing of its adjacency matrix, and the eigenvalues of this cover are  the union of the eigenvalues of the original graph and of the signed matrix. Last year, Hao Huang resolved the Sensitivity Conjecture by proving that every $(2^{n-1}+1)$ vertex induced subgraph of $Q_n$ has maximum degree at least $\sqrt{n}$. His proof relies on constructing a $\pm1$ signing of the adjacency matrix of $Q_n$ which specifies exactly the cover $\widetilde{Q}_n$ described above (see Section \ref{cttohh}). Motivated by this connection, we develop a natural framework for both results. This framework extends to covers of larger index, or equivalently, adjacency matrices whose ``signings'' take values in larger cyclic groups represented by complex roots of unity.  As in the $\pm1$ signing case, the spectra of these complex adjacency matrices encode the new eigenvalues of our covering graphs. Our key definition is that of a \textbf{two-eigenvalue cover}, (or 2ev-cover) which is a covering graph whose adjacency spectra differs (as a multiset) from that of the graph it covers by exactly two eigenvalues. It is worth noting that the only symmetric or hermitian matrices with exactly one eigenvalue are scalar multiples of the identity matrix. Consequently, there is no analogous notion of a ``1ev-cover,'' so we are studying the extremal case.
 
2ev-covering graphs occur frequently in the study of distance-regular graphs, where this two eigenvalue condition is necessary for a lifted graph to retain the distance-regularity of its quotient. However, this condition is not sufficient: $Q_n$ is distance-regular, while $\widetilde{Q}_n$ is not. This begs the question of which combinatorial properties are necessarily preserved by 2ev-covers.  By employing ideas of Godsil and Hensel \cite{GH} we answer this question when the structure of the cover in question is governed by an abelian group. In particular we show walk regularity, a weaker version of distance-regularity, is preserved in abelian two eigenvalue covers. We then show that any cyclic 2ev-cover of the complete graph is a distance-regular and antipodal. Any cyclic 2ev-cover of a complete bipartite graph is distance-regular with diameter four. More generally, we give a necessary and sufficient condition for a cyclic 2ev-cover of a strongly regular graph to be distance-regular. We conclude with some examples and directions for future work.

The theory of covering graphs was first formally developed in 1974 by Biggs \cite{NB} and independently by Gross \cite{JG}. Biggs devoted a chapter of his book to setting down a categorical definition and explicating several interesting examples. Meanwhile, Gross developed an analogous theory as a tool for studying graph embeddings. This topological approach was echoed by Cohen and Tits \cite{CT}, who described $\widetilde{Q}_n$ as the cover corresponding to a particular index 2 subgroup of the fundamental group of the (1-skeleton of) the hypercube. 

Even before these formal treatments, many interesting covering graphs were known to mathematicians studying distance-regular graphs. This connection can be described concretely by a result of Smith which states that every distance-regular graph is either primitive, bipartite, or antipodal (and hence a cover). See \cite{BCN}, Chapter 4. In 1990, Godsil and Hensel developed the theory of the distance-regular antipodal covers of the complete graph (drackn's). They described the relation between such covers and regular two-graphs, as well as a connection to sets of equiangular lines in euclidean space. \cite{GH}, \cite{LS}. This work was extended by Coutinho et. al. \cite{CGSZ}  to relate drackns to sets of complex lines, and by Iverson and Mixon \cite{IM}, to a broader class of lines called roux.

Backtracking to 1987, Gross and Tucker \cite{GT} expanded Gross's work and defined the permutation voltage graphs, from which all covering graphs can be constructed.  Zaslavsky  \cite{TZ} studied these graphs from a matroidal perspective under the name gain graphs. This study has blossomed in many different directions whose references are recorded by Zaslavsky in the dynamic bibliography \cite{TZ2}. One recent example is a paper of Cavaleri, D'Angeli, and Donno \cite{CDD}, wherein the authors answer a question of Zaslavsky on the algebraic characterization of balanced gain graphs.  In the final section of this paper, the authors call upon the connection between gain graphs and covering graphs to translate their results to the latter setting. In particular, they show that a certain condition on the balancedness of cycles is equivalent to the associated cover being isomorphic to a disjoint union of copies of the underlying graph. This is closely related (but of an opposite flavor) to our Theorem \ref{wreg1} on the walk regularity of 2ev-covers.  Cavaleri, D'Angeli, and Donno have given a spectral characterization of the covering graphs that are in some sense trivial while we are interested in those that have interesting combinatorial structure. 

Running parallel to the inquiries above, the spectral theory of $\pm1$ signings of adjacency matrices has been well developed. The recent advances in this area have, at least partially, been inspired by Bilu and Linial's \cite{BL} proposed construction of expander graphs via an iterative sequence of 2-fold covers, and by Marcus, Spielman, and Srivastava's \cite{MSS} great success in implementing this proposal. Interest has again been renewed by Hao Huang's \cite{HH} resolution of the sensitivity conjecture, which is predicated on signed adjacency matrices with exactly two distinct eigenvalues. Even prior to Huang's proof, the taxonomy of two-eigenvalue signed graphs had begun to emerge, see \cite{ZS}, \cite{GF},\cite{MS}. Moreover, see \cite{BCKW} for a survey of results on two-eigenvalue signed graphs and some interesting open problems.

\section{Structure of covering graphs}
We begin with a connected undirected graph $Y=(V,E)$ and a permutation group $G$. Define a \textbf{symmetric arc function} to be a function $f: V\times V \rightarrow G\cup\{0\}$ for which $f(u,v)=f(v,u)^{-1}$ whenever $f(u,v)\in G$ and $f(u,v)=0$ if and only if $\{u,v\}\notin E(Y)$. The pair $(Y,f)$ is called a \textbf{gain graph}, and can be used to define a cover of $Y$ as follows. Let $R=\{1,2,\dots, r\}$ be a set on which $G$ acts. The \textbf{cover} $Y^f$ of $Y$ is the graph with vertex set $V\times R$ and $(v,j)\sim_{Y^f}(u,k)$ exactly when $v\sim_Yu$ and $f(u,v)j=k$. 

Many structural characteristics of $Y^f$ are immediate. If $Y$ is a regular graph, then $Y^f$ is regular with the same valency. This can be considered as a prototypical example of the ``combinatorial regularity'' of covering graphs that we are interested in studying. 

There exists a graph homomorphism $\gamma:Y^f\rightarrow Y$ defined by $\gamma(v,j)\mapsto v$; it is called the \textbf{covering map}. The preimages $\gamma^{-1}(v)$ are called the \textbf{fibers} of $Y^f$, and they induce cocliques (edgeless subgraphs). The subgraph induced by any two fibers $\gamma^{-1}(u),\gamma^{-1}(v)$ is a perfect matching if $\{u,v\}\in E$ and is a coclique otherwise.

Two distinct symmetric arc functions can yield isomorphic covering graphs. In particular, permuting the vertices of each fiber $\gamma^{-1}(u)$ of $Y^f$ by some $\sigma_u$ in the symmetric group $\sym(r)$ gives rise to an isomorphic cover $Y^g$ with symmetric arc function $g(i,j)=\sigma^{-1}_i f(i,j)\sigma_j$. Whenever we choose these permutations so that $g(i,j)=id$ for all edges $\{i,j\}$ of some spanning tree of $Y$ we call $g$ a  \textbf{normalized} symmetric arc function.

\section{Representations of finite groups}
Let $Y^f$ as above with $f$ a normalized arc function. Denote by $\langle f \rangle$ the subgroup of $\sym(r)$ generated by the image of $f$. The linear representations of this group are a valuable tool in studying the cover $Y^f$. To keep this paper self contained we will briefly mention the basic representation theory that is relevant to our work. The following discussion is essentially identical to (a subset of) that in Sections 8 and 9 of \cite{GH}. A more detailed account can be found in any text on the representation theory of finite groups, e.g, \cite{JL}. 
  
  A \textbf{representation} $\phi$ of a group $G$ is a homomorphism into the general linear group $GL(r,\bb{C})$. We say that $r$ is the \textbf{degree} of the representation. For each $A\in GL(r,\bb{C})$ the map $\phi^A(g)$ defined by \[\phi^A(g)=A\phi(g)A^{-1}\] is also a representation of $G$, and is said to be \textbf{equivalent} to $\phi$. The \textbf{trivial} representation is the map from $G$ onto the identity matrix $I_r$. If $\phi$ and $\psi$ are representations of $G$ with degrees $r$ and $s$ respectively then the map $\phi+ \psi$ defined by \[g\mapsto \pmat{\phi(g)&0\\0&\psi(g)}\] is a representation of degree $r+s$. We say that $\phi+ \psi$ is the sum of $\phi$ and $\psi$. The set of representations of $G$ is closed under non-negative integral linear combinations. A representation $\phi$ is \textbf{irreducible} if there is no nontrivial subspace of $\bb{C}^r$ which is invariant under $\phi(g)$ for all $g\in G$. It can be shown that a representation of a finite group is reducible if and only if it is equivalent to a nontrivial positive integral linear combination of representations of $G$.

Define a vector space $\mc{V}$ to be the $\bb{C}$-span of the elements of $G$ and define a $G$-action on this vector space where each $g\in G$, acts on the basis vectors by left multiplication (and the action is extended linearly). In the basis of group elements of $G$, each $g$ acts as a permutation matrix $\rho(g)$. The map $\rho: G\rightarrow GL(|G|,\bb{C})$ is called the (left) \textbf{regular representation} of $G$.

  \begin{theorem}
  Let $G$ be a finite group. Then
 \begin{enumerate}
\item[(a)] $G$ has finitely many inequivalent irreducible representations $\phi_i$ for $(i=0,1,\dots)$.
\item[(b)] If $r_i$ denotes the degree of $\phi_i$ and $\rho$ is the regular representation of $G$ then $\rho$ is equivalent to $\sum_{i}r_i\phi_i$.
\item[(c)] If $G=H_1\times H_2$, then the regular representation of $G$ is the tensor product of the regular representations of $H_1$ and $H_2$.
\item[(d)]  If $G$ is abelian then each $r_i=1$ and $\phi_i(g)$ is a (not necessarily primitive) $m$th root of unity where $m$ is the order of $G$.\qed

 \end{enumerate}
  \end{theorem}

We have described above how to construct $Y^f$ from $(Y,f)$. When $\langle f\rangle$ acts regularly on the vertices of a fiber, the regular representation of $\langle f \rangle$ encodes this process as follows.  Let $A=A(Y)$ denote the adjacency matrix of $Y$, and suppose  $\phi$ is a degree $r$ representation of $\langle f \rangle$, define $A^{\phi(f)}$ to be the matrix obtained by replacing each non-zero entry $A_{u,v}$ with the matrix $\phi(f(u,v))$ and each zero entry with the $r\times r$ matrix of zeros.

 \begin{theorem}\label{rep1} (\cite{GH} Sections 7, 8)
 If $f$ is a normalized arc function then $\langle f \rangle$ acts regularly on the vertices of each fiber of $Y^f$ if and only if $A(Y^f)=A(y)^{\rho(f)}$.\qed
 \end{theorem}

If $Y^f$ is connected then $\langle f \rangle$ acts transitively on $\{1,\dots, r\}$ and $|\langle f \rangle|=r\cdot |stab(x)|$ for any $x\in \{1,\dots r\}$. If $\langle f \rangle$ is abelian we may quotient out this stabilizer, considering the resulting abelian subgroup acting regularly on $\{1,\dots r\}$. In this paper we will be concerned with abelian covers, and we will thus assume that our groups $\langle f \rangle$ act regularly on $\{1,\dots, r\}$. Hence we will assume that $A(Y^f)=A(Y)^{\rho(f)}$.
 
This issue is more subtle if $\langle f \rangle$  is not abelian: We certainly may always work with the regular representation of $\langle f \rangle$, but we will often miss out on certain covers of small index in the process. This is essentially the distinction between ordinary voltage graphs and permutation voltage graphs as explicated by Gross and Tucker \cite{GT}. See Section \ref{nonab} for an example of a cover that is not regular.

 \section{Two eigenvalue covers}
Since $\rho$ is equivalent to $\sum_{i=0}^{|\langle f \rangle |}\phi_i$, it follows that $A(Y)^\rho$ is equivalent to a block diagonal matrix whose blocks are the $A(Y)^{\phi_i(f)}$. The first of these blocks is $A(Y)^{\phi_0(f)}=A(Y)$, from which we deduce that the spectrum of $A(Y)$ is a subset of the spectrum of $A(Y^f)$. Let $\mathrm{Spec}(A)$ denote the spectrum of $A$ as a multi-set. 
 
We say $Y^f$ is a \textbf{two-eigenvalue cover} (2ev-cover) if $\mathrm{Spec}((A(Y^f))\setminus \mathrm{Spec}((A(Y))$ has exactly two distinct elements.  Suppose $Y^f$ is a 2ev-cover, then each $A(Y)^{\phi_i(f)}$ for $i>1$ must be roots of the same minimal polynomial (of degree 2). Our primary means of studying 2ev-covers is via these $A(Y)^{\phi_i(f)}$.
 
\begin{prop} 
Let $\phi$ be some non-trivial irreducible representation of $\langle f \rangle$ and let $Y^f$ be a cover for which $S=A(Y)^{\phi(f)}$ has exactly two distinct eigenvalues $\theta$ and $\tau$. Then there exist real numbers $\lambda, \mu$ so that  \[S^2=\lambda S+\mu I,\] Moreover, $Y$ is regular with valency $\mu=-\theta\tau$.
\end{prop}
 
\begin{proof}
The minimal polynomial of $S$ has degree 2 and may be written as $x^2-\lambda x-\mu$. Since $f$ is a symmetric arc function, $S$ is hermitian, and has real eigenvalues $\theta, \tau$. We have $\mu=-\theta\tau$ and $\lambda=-\theta-\tau$. Moreover $S$ satisfies the equation $S^2=\lambda S+\mu I$. Since $Y$ is loopless, $S$ has zero diagonal. Finally, the diagonal entries of $S^2$ are equal to the diagonal entries of $A(Y)^2$ which are the valency of the vertices of $Y$.\qed
\end{proof}
 
 \textbf{Remark:}  In the event that $\langle f \rangle$ is abelian, its irreducible representations have degree 1, and for any non-trivial irreducible representation $\phi$, the matrix $S=A(Y)^{\phi(f)}$ is particularly familiar. Here $S$ is the group valued adjacency matrix of a complex unit-gain graph. This is a direct generalization of the signed adjacency matrix, a setting where the 2ev condition has already been extensively investigated.

\section{Covers of walk regular graphs}

A graph is \textbf{walk regular} if the number of closed walks of length $k$ starting from (and ending at) vertex $v$ depends only on $k$. Equivalently, a graph with adjacency matrix $A$ is walk regular if $A^k$ has constant diagonal for all positive integers $k$. Walk regular graphs have been studied due to their ``spectral regularity:'' The subgraphs $W/v$ have identical characteristic polynomials for all $v\in V(W)$, (see \cite{GM}). Moreover, it can be shown that all distance-regular graphs are walk regular. In Section 6 we will be able to say a fair amount about cyclic two-eigenvalue covers of distance-regular graphs. Working instead with walk regularity we can say more.

\begin{theorem}\label{wreg1} If $Y$ is a walk regular graph and $X=Y^f$ is a cyclic 2ev-cover of $Y$ then $X$ is walk regular.\end{theorem}

\begin{proof}
Let $\phi_0,\dots \phi_{r-1}$ be the irreducible representations of $\langle f \rangle$, with $\phi_0$ the trivial representation, and let $S_1=A(Y)^{\phi_1(f)}$.  
Let $A_g$ denote the 01 matrix \[(A_g)_{u,v}=\begin{cases} 1& \text{if } f(u,v)=g \\ 0& \text{otherwise}\end{cases}.\]

\noindent Identifying our index set with the (additive) cyclic group on $\{0,\dots,r-1\}$ we may write \[S_1=\sum_{i=0}^{r-1} \omega^iA_i\] where $r$ is the index of the cover, and $\omega$ is a primitive $r$th root of unity. By replacing $\phi_1$ with $\phi_{i}$ for $i\in \{2,\dots r-1\}$ we exchange $\omega$ for the other $r$th roots of unity and obtain matrices $S_j=A(Y)^{\phi_j(f)}$ defined by \[S_j=\sum_{i=0}^{r-1} \omega^{j\cdot i}A_i.\]

\noindent Let $\circ$ denote the entry-wise matrix product, and let $\ell$ be a non-negative integer.

\textbf{Claim.} \[A(X)^\ell\circ I_{nr}=(\frac{1}{r}(\sum_{j=0}^{r-1}S_j^\ell)\otimes I_r)\circ I_{nr}\]

This will complete our proof since each of the individual $S_j^\ell$ terms have constant diagonal:  $S_1$ has the same support as $A$, so $S_1$ has zero diagonal.  By Proposition 4.1, any power of $S_1$  is polynomial in $\{I,S_1\}$, hence $S_1^\ell$ has constant diagonal. As we have noted in Section 4, the matrices $S_j$ for $j\geq 1$ must be roots of the same minimal polynomial. So each $S_j$ is a root of the minimal polynomial of $S_1$, and $S_j^\ell$ has constant diagonal as well. By assumption $Y$ is walk regular, so $S_0^\ell=A(Y)^\ell$ has constant diagonal. Hence the diagonal of $A(X)^\ell$ is the diagonal of a sum of matrices with constant diagonal and we are done.

\textbf{Proof of Claim.} From Theorem \ref{rep1} we have \[A(X)=\sum_{i=0}^{r-1}A_i\otimes \rho(\omega^i).\] So $A(X)^\ell$ is a sum of products of the form \[(A_{m_1}A_{m_2}\dots A_{m_\ell})\otimes \rho( \omega^{\sum_k m_k})\] where $m_i\in \{0,\dots r-1\}.$  Since $\langle f \rangle$ is regular, the only summands that have non-zero diagonal entries are those for which the righthand tensor factor is the identity i.e., those terms where $\sum m_k=0 \mod r$. Let $\mc{M}$ be the set of $\ell$-tuples $(m_1,\dots m_\ell)$ of elements of $\langle f \rangle$, and let $\mc{M}_0$ the subset of $\mc{M}$ for which  $\sum m_k=0\mod r$. The previous remark shows that $A(X)^\ell$ has the same diagonal as 

\begin{equation} \sum_{M\in\mc{M}_0}\big(\prod_{m_i\in M}A_{m_i}\big)\otimes I_r=\Big(\sum_{M\in\mc{M}_0}\prod_{m_i\in M}A_{m_i}\Big)\otimes I_r\end{equation}

\noindent On the other hand, notice that \[S_j^\ell=(\sum_{i=0}^{r-1} \omega^{j\cdot i}A_i)^\ell=\sum_{M\in\mc{M}}\prod_{m_i\in M}\omega^{j\cdot m_i}A_{m_i}\]

\noindent Now, summing over all such $S_j$ we have \[\sum_{j=0}^{r-1} S_j^\ell=\sum_{j=0}^{r-1} \sum_{M\in\mc{M}}\prod_{m_i\in M}\omega^{j\cdot \sum m_i}A_{m_i} =\sum_{M\in\mc{M}}\sum_{j=0}^{r-1}\prod_{m_i\in M}\omega^{j\cdot \sum m_i}A_{m_i} \]

\noindent For any $M\in \mc{M}\backslash \mc{M}_0$ the $\omega^j$ term is independent of the product, hence \[\sum_{j=0}^{r-1}\prod_{m_i\in M}\omega^{j\cdot \sum m_i}A_{m_i}=\sum_{j=0}^{r-1}\omega^j\prod_{m_i\in M}\omega^{\sum m_i}A_{m_i}=0.\] And so \begin{equation}\sum_{j=0}^{r-1} S_j^\ell=\sum_{M\in\mc{M}_0}\sum_{j=0}^{r-1}\prod_{m_i\in M}\omega^{j\cdot \sum m_i}A_{m_i} =\sum_{M\in\mc{M}_0}\sum_{j=0}^{r-1}\prod_{m_i\in M}A_{m_i}=r\sum_{M\in\mc{M}_0}\prod_{m_i\in M}A_{m_i}.\end{equation}
Putting together equations $(1)$ and $(2)$ proves the claim.\qed
\end{proof}

\begin{corollary} If $Y$ is a walk regular graph and $X$ is an abelian 2ev-cover of $Y$ then $X$ is walk regular.\end{corollary}

\begin{proof}
The regular representation of the abelian group $G=H_1\times H_2$ is the tensor product of the regular representations of $H_1$ and $H_2$. Hence for each $g\in G$ we have $\rho(g)=\rho_{H_1}(h_1)\otimes \rho_{H_2}(h_2)$, and $\rho(g)$ still has zero diagonal unless $g=id$. So the only terms in the expansion of $(\sum_{g\in G} A_g\otimes \rho(g))^\ell$ that contribute to the diagonal are those whose first tensor factor has subscripts whose sum is zero mod $|G|$. The irreducible representations of $\rho$ are still of degree one and the argument  proceeds exactly as above. \qed
\end{proof}

\section{Covers of strongly regular graphs}
Let $X$ be a connected regular graph and $v\in V(X)$. The \textbf{distance partition} of $X$ with respect to  $v$ is a partition of $V(X)$ into cells $\{\Gamma_0(v),\Gamma_1(v),\dots,\Gamma_d(v)\}$ where $\Gamma_i(v)$ consists of the vertices at distance $i$ from $v$.  A partition $\pi=(\pi_1,\dots,\pi_k)$ of $V(X)$ is \textbf{equitable} if for each pair $(i,j)\in \{0,\dots d\}\times \{0,\dots d\}$ the number of edges of $X$ with one end in $\pi_i$ and the other end in $\pi_j$ depends only on $i$ and $j$.

A graph whose distance partition is equitable is \textbf{distance-regular}. Each distance-regular graph of diameter $d$ has an \textbf{intersection array} $\{b_0,b_1,\dots b_{d-1};c_1,c_2\dots c_d\}$ where $b_i$ is the number of neighbors $u\in \Gamma_{i+1}$ of any $v \in \Gamma_i$ and $c_i$ is the number of neighbors $u\in \Gamma_{i}$ of any $v \in \Gamma_{i+1}$. 

A distance-regular graph with diameter 2 is \textbf{strongly regular}. The \textbf{parameters} of a strongly regular graph are the 4-tuple $(n,k,a,c)$ where $n$ is the number of vertices, $k$ is the valency, $a:=k-b_1-c_1$ is the number of common neighbors of a pair of adjacent vertices and $c:=c_2$ is the number of common neighbors of a pair of non-adjacent vertices. See Chapter 10 of \cite{CG2} for a thorough introduction.

A graph $Y$ of diameter $d$ is \textbf{antipodal} if ``$v$ is at distance $d$ from $u$'' is an equivalence relation on $V(Y)$. If $X$ is antipodal and distance-regular, we consider the graph $Y$ whose vertices are the antipodal classes with two classes adjacent if there is an edge of $X$ with one end in each class. $Y$ is called the \textbf{antipodal quotient} of $X$, and $X$ is a cover of its antipodal quotient. See Theorem 2.1 of \cite{GH}, \cite{BCN}, \cite{AG}. In the case that this antipodal quotient $Y$ is the complete graph, $X$ is a distance-regular antipodal cover of $K_n$, or \textbf{drackn} for short.\label{drackn}

Distance-regular graphs are a well studied topic in algebraic graph theory. They are related to combinatorial designs, linear codes, association schemes, and orthogonal polynomials. The literature is vast, but a good starting place would be \cite{BCN},\cite{NB},\cite{CG2}. 

We take a quick detour to offer a few interesting examples.
\begin{enumerate}

\item The hypercube is distance-regular and antipodal with antipodal classes of size 2. Its antipodal quotient is also distance-regular. In the case of $Q_3$ this antipodal quotient is the complete graph $K_4$, and $Q_3$ is a drackn. 
\item The Hoffman-Singleton graph $H$ is strongly regular with parameters (50,7,0,1), For any fixed vertex $v\in V(H)$ the subgraph induced by vertices at distance 2 from $v$ is a 6-fold distance-regular antipodal cover of $K_7$.
\item The Johnson graph $J(n,k)$ has as its vertices the $k$ element subsets of $\{1,\dots, n\}$. Two vertices are adjacent if  their intersection has size $k-1$. Any Johnson graph is distance-regular.
\end{enumerate}

Now we will give a characterization of when a cyclic 2ev-cover of a strongly regular graph is distance-regular. 

Our key lemma shows that the 2ev condition forces the arc function to take a very particular form.

\begin{lemma}\label{ML}
Let $Y$ be a connected distance-regular graph with parameters $a:=k-b_1-1$ and $c:=c_2$. Let $\phi$ a non-trivial irreducible representation of $\langle f \rangle$, $S=A(Y)^{\phi(f)}$, and $x^2-\lambda x-\mu$ the minimal polynomial of $S$. Assume that $f$ is normalized so that $S$ is of the form \[S=\pmat{0&\bb{1}&\bb{0}&\dots\\\bb{1}^T&N_1 &B\\ \bb{0}^T&B^T & N_2&\\\vdots& &&\ddots}.\] There exist constants \[t=\frac{a-\lambda}{r},~~s=\frac{c}{r}\] so that \begin{itemize}\item For each $i\in \{1,\dots, r-1\}$, each column of the submatrix $N_1$ contains exactly $t$ entries equal to $\omega^i$. \item For each $j\in \{0,\dots, r-1\}$ each column of the submatrix $B$ contains exactly $s$ entries equal to $\omega^j$.\end{itemize}
\end{lemma}

\begin{proof}
For a submatrix $M$ of $S$, let $C_m(M)$ denote the sum of the entires of column $m$ of $M$. Matrix multiplication shows that the first row of $S^2$ is \[\pmat{k&C_1(N_1)&\dots& C_k(N_1)&C_1(B)&\dots&C_{\ell}(B)&0&\dots&0}.\] 
On the other hand, $S^2 = \lambda S + \mu I$ so the first row of $S^2$ is \[\pmat{\mu& \lambda &\dots& \lambda& 0 &\dots& 0}\]

We have three immediate consequences. \begin{itemize}
\item $\mu=k$. 
\item For each $1\leq m\leq k, C_m(N_1)=\lambda$. 
\item For each  $1\leq n \leq \ell$, $C_n(B)=0.$
\end{itemize}

For some fixed column $m'$ of $N_1$, let $t_i$ denote the number of occurrences of $\omega^i$ in that column. We may write $C_m'(N_1)=\sum_{i=0}^{r-1}t_i\omega^i$. Note that $S$ is hermitian and that $\lambda$ is (minus) the sum of the two distinct eigenvalues of $S$, hence $\lambda$ is real. So the above expression for $\lambda$ must be symmetric under any permutation of the $rth$ roots of unity which fixes the rationals. It follows that $t_1=t_2=\dots =t_{r-1}$. Denote this common value by $t$. We have \[\lambda=C_{m'}(N_1)=t_0+t\Big(\sum_{i=1}^{r-1}\omega^i\Big)=t_0-t.\] Since $Y$ is distance-regular, we have $\sum_{i=0}^{r-1}t_i=a$ which we may now write as \[t_0+t(r-1)=a.\] Combining our equations we obtain \[t=\frac{a-\lambda}{r}.\] Note that $t$ is independent of the chosen column $m'$, this proves the first claim.

For some fixed column $n'$ of $B$, let $s_j$ denote the number of occurrences of $\omega^j$ in that column. We play the same game with the equations \[C_{n'}(B)=\sum_{j=0}^{r-1}s_j\omega^j=0,~~~~~\sum_{j=0}^{r-1}s_j=c\] and find that \[s:=s_0=s_1=\dots=s_{r-1}=\frac{c}{r}.\] is independent of $n'$. This proves the second claim. \qed
\end{proof}

Now we will employ this lemma to characterize cyclic 2ev-covers of strongly regular graphs. We begin with the degenerate case.

\begin{theorem}\label{drg1}
If $Y$ is the complete graph on $n$ vertices and $Y^f$ is a connected cyclic 2ev-cover of $Y$, then $Y^f$ is a drackn.
\end{theorem}

\begin{proof}
Note that $t>0$, since $t=0$ implies $Y^f$ is the disjoint union of $r$ copies of $K_n$.

Label the vertices of $Y$ as $v_0,\dots, v_{n-1}$, and let $\gamma$ be the covering map from $Y^f$ to $Y$. For each fiber $\gamma^{-1}(v_i)$  let $(v_i,0),\dots, (v_i,r-1)$ denote the vertices of that fiber. Let $v_0$ be the vertex of $Y$ indexing the first row and column of $S$, and choose (arbitrarily) some vertex of $\gamma^{-1}(v_0)$ to be $(v_0,0)$. We consider the distance partition of $Y^f$ with respect to $(v_0,0)$.
 
$Y^f$ has at most 4 distinct eigenvalues ($n-1$, $-1$, and the eigenvalues of $S$), hence it has diameter at most $3$. Any two vertices of $Y^f$ from the same fiber are not adjacent and do not have any common neighbors, hence the diameter of $Y^f$ is exactly 3. The assumption that $S$ is normalized along the first row and column is equivalent to the claim that \[\Gamma_{1}(v_0,0)=\{(v_i,0):i\in \{1,\dots n-1\}\}.\] 

For any $j\neq 0$ and  $(v_j,q)$, the result of Lemma \ref{ML} and the fact that $t>0$ imply that, in the  $j$th column of $S$ there is some entry $S_{i,j}=\omega^q, i\neq 0$. Hence $(v_j,q)\sim(v_i,0)$ and the distance between $(v_0,0)$ and $(v_j,q)$ is at most (in fact, exactly) 2.  It follows that each vertex that is neither a neighbor of $(v_0,0)$ nor in $\gamma^{-1}(v_0)$ is at distance two from $(v_0,0)$. Moreover for any such vertex  $(v_j,q)$ there are exactly $t$ entries of the $j$th column equal to $\omega^q$. So $(v_0,0)$ and $(v_j,q)$ have exactly $t$ common neighbors. It now follows from Lemma 3.1 in \cite{GH} that $X$ is an $(n,r,t)$-drackn.  \qed
\end{proof}

Now we consider the case where $Y$ is strongly regular. As illustrated by Example 7.1, it is not true that all cyclic 2ev-covers of strongly regular graphs are distance-regular, but we can give an exact spectral characterization of those which are.

\begin{theorem}\label{drg2} Let $Y$ be a connected strongly regular graph with parameters $(n,k,a,c)$, $\phi$ a non-trivial irreducible representation of $\langle f \rangle$, $S=A(Y)^{\phi(f)}$, and $x^2-\lambda x-\mu$ the minimal polynomial of $S$. Then $Y^f$ is distance-regular if and only if $a=\lambda$.
\end{theorem}

\begin{proof}
Suppose $a=\lambda$. Then $t=0$ and  $N_1$ is a 0-1 matrix. We construct the distance partition of $Y^f$ with respect to $(v_0,0)$ and find that the neighborhood of $(v_0,0)$ induces a subgraph isomorphic to $Y[N_1]$, call this subgraph $H_0$, and note that, since the arc function is trivial on $N_1$, there are $r-1$ other subgraphs $H_1,\dots H_{r-1}$ contained in $y^f$, each of which is the neighborhood of some $(v_0,j)$ for $j\in \{1,\dots r-1\}$. Moreover, there are no edges between any of the distinct $H_i$ (again, since the arc function is trivial on $N_1$). 

Since $Y$ is connected, we have $c>0$. By Lemma \ref{ML}, each column of $B$ contains $\frac{c}{r}$ entries equal to each $\omega^i$. Hence each vertex $(v_j,i)$ for $v_j\in N_2$ is adjacent to $\frac{c}{r}$ vertices of each $H_i$. This shows that  $\Gamma_2(v_0,0)$ is the union of fibers of vertices in $N_2$, and $\Gamma_3(v_0,0)$ is the union of $H_i$ for $i>0$. Finally,  $\Gamma_4(v_0,0)$ is the remaining $r-1$ vertices of the $v_0$ fiber. It follows that the distance partition is equitable, with intersection array \[\{k,k-a-1,\frac{c(r-1)}{r},1;1,\frac{c}{r},k-a-1,k\}.\]

Now suppose $a\neq\lambda$. So $t>0$ and each column of $N_1$ other than the first contains each $\omega^i$ for $i\in \{0,\dots, r-1\}$. Since $Y$ is connected, $c>0$, and each column of $B$ also contains at least one entry equal to each $\omega^i$.

As in the proof of Theorem \ref{drg1}, this means that when we construct the distance partition of $Y^f$ with respect to $(v_0,0)$, we have \[\Gamma_1(v_0,0)=\{(v_j,0) : j\in \{1,\dots, n-1\}\} \]  and \[\Gamma_3(v_0,0) \supseteq \{(v_0,j): j \in \{1,\dots r-1\}\}.\] 
There are two types of vertices remaining. Those of the form $(v_\alpha,i)$ for $v_\alpha$ a neighbor of $v_0$ and those of the form $(v_\beta,i)$ for $v_\beta$ at distance 2 from $v_0$.  By Lemma \ref{ML}, each $v_\alpha$ is adjacent to $t>0$ neighbors of $v_0$ for which the edge joining them is assigned the value $\omega^j$. Since this holds for all $j\in\{1,\dots, r-1\}$ each $(v_\alpha,i)$ for $i>0$ is at distance 2 from $(v_0,0)$. Similarly, by Lemma \ref{ML} all vertices of the fiber $\gamma^{-1}(v_\beta)$ are at distance two from $(v_0,0)$, and none of these vertices are incident with any vertex of $\gamma^{-1}(v_0)$, so some vertices in $\Gamma_2$ have neighbors in $\Gamma_3$ while others do not, and $X$ is not distance-regular.\qed\end{proof}

As a special case of the previous theorem, we show that every cyclic 2-ev cover of a complete bipartite graph is distance-regular.

\begin{corollary}
Suppose $Y$ is complete bipartite and $Y^f$ is a cyclic $r$-fold 2-ev cover of $Y$. Then $Y=K_{n,n}$ with $r|n$, and $X$ is bipartite distance-regular with diameter 4. 
\end{corollary}

\textbf{Proof}. It follows from Proposition 4.1  that $Y$ is regular. 
 Hence $Y$ is strongly regular with parameters $(2n,n,0,n)$ for some $n$. Normalizing as in Theorem \ref{drg2} and noting that $a=0$ we have 
\[S=\pmat{0&\bb{1}&\bb{0}\\\bb{1}^T&\bb{0}_{n\times n} &B\\ \bb{0}^T&B^T &N_2}\]

Hence $\lambda=0$ as well, and $Y^f$ is distance-regular. Moreover, each column of $B$ contains $\frac{c}{r}=\frac{n}{r}$ entries equal to $\omega^i$, hence $r|n$. It is well known that a graph is bipartite if and only if its adjacency spectrum is symmetric about 0. $Y$ is bipartite, and $S$ has spectrum symmetric around zero. Hence $Y^f$ has spectrum symmetric around zero and is bipartite.\qed

\section{Examples}
In this section we provide a few examples of 2ev-covers. 

\subsection{The Cohen-Tits cover of $Q_n$}\label{cttohh} We wish to construct a 2-fold cover of $Q_n$ which contains no 4-cycles. We employ the inductive construction of $Q_n$ as two copies of $Q_{n-1}$ joined by a perfect matching. 
Consider a $\pm1$ signing of the edges of $Q_{n-1}$, or equivalently, a symmetric arc function \[f:V\times V \rightarrow \bb{Z_2}\cup\{0\}.\]  The condition that the resulting 2-fold cover $Q_{n-1}^f$ has no 4-cycles  is equivalent to the condition that each 4-cycle of $Q_n$ has an odd number of edges $e$ with $f(e)=-1$. It follows that the opposite signing, $f'=-f$ will also give rise to a 4-cycle free 2-fold cover. Joining these oppositely signed copies of $Q_{n-1}$ by a perfect matching whose edges are all given the same sign (say, $+1$) yields a signing of $Q_n$ in which each 4-cycle has an odd number of negative edges, and thus gives rise to the desired cover.

Performing this process at the level of the adjacency matrix is exactly the construction given by Huang in \cite{HH}. Define \[A_1=\pmat{0&1\\1&0},~~~A_{n}=\pmat{A_{n-1}&I_{n-1}\\I_{n-1}&-A_{n-1}}.\] Then taking $f(u,v)=(A_n)_{u,v}$ defines a symmetric arc function on the hypercube $Q_n$. The resulting cover of index 2 is the unique girth six 2-fold cover of $Q_n$ determined by Cohen and Tits in \cite{CT}. Note that this cover is not distance-regular, however its distance distribution diagram is a tree which is not too far away from being a path. (See \cite{BCN} section 9.2E).

\subsection{Folded cubes}
The folded $n$-cube, denoted $\square_n$ is the antipodal quotient of $Q_{n}$. Equivalently, $\square_n$ is the graph obtained from $Q_{n-1}$ by adding a perfect matching between antipodal vertices. In \cite{BCN} the authors remark, without proof, that a 4-cycle free 2-fold cover of $\square_n$ exists if and only if $n$ is congruent to $0, 1$ mod 4 and $n>4$. (See \cite{BCN} section 9.2E). These covers are necessarily 2ev. A construction of these covers, as well as a proof of the non-existence of such covers for the other congruence classes can be found in the recent arxiv preprint of Alon and Zheng \cite{AZ}. 

Alon and Zheng have constructed what they call unitary signings of Cayley graphs for elementary abeliean 2-groups. The graphs in question are the folded $d$-cubes and the cartesian products of folded $(k+1)$-cubes with $(d-k)$-cubes. And unitary signings can be viewed as symmetric arc functions on the underling graphs which give rise to cyclic covers of index 4 and girth $>4$.

\subsection{The 3-fold cover of $Kn(7,2)$}
The Kneser graph $Y=Kn(7,2)$ has as vertices the 21 two element subsets of $\{0,1,\dots 6\}$. Two such subsets are adjacent if they are disjoint. Its spectrum is \[\{10^{(1)},1^{(14)},-4^{6}\}.\] There is a unique antipodal distance-regular graph $X$ of diameter 4 on 63 vertices. It's spectrum is \[\{10^{(1)},5^{(12)},1^{(14)},-2^{(30)},-4^{6}\}.\] As suggested by the spectrum, $X$ is a 3-fold 2ev-cover of $Y$, (See \cite{BCN} 13.2B).

\subsection{Drackn's and dracknn's}

As seen in Section 6, the drackn's are precisely the cyclic 2-ev covers of the complete graph, and the distance-regular antipodal covers of the complete bipartite graph (dracknn's) are precisely the cyclic 2-ev covers of $K_{n,n}$. In particular, the distance-regular 2-fold covers of $K_{n,n}$ are the Hadamard graphs, (See \cite{BCN} Section 1.8). More generally, if $H$ is a complex Hadamard matrix of Butson type (meaning the entries are $q$th roots of unity), with \[B=\pmat{0& H\\H^T&0}\] then $f(u,v)=B_{u,v}$ defines a symmetric arc function which gives rise to a distance-regular 2ev-cover of $K_{q,q}$.

\subsection{A non-abelian 2-ev cover}\label{nonab}
Many of our proof techniques are predicated on the cyclic or abelian nature of the group $\langle f \rangle$. In Theorem \ref{wreg1}, It was important that $\langle f \rangle$ acted regularly $\{1,\dots, r\}$. In Theorems \ref{drg1} and \ref{drg2}, we made (implicit) use of the fact that the irreducible representations of $\langle f \rangle$ were one-dimensional.

However, we do not lose all of our tools when considering non-abelian covers. If we work with any permutation representation $\tau$ for $\langle f \rangle$, the the spectrum of the resulting cover will decompose as the union of the spectra of the covers defined by the irreducible representations of $\langle f \rangle$. The regular representation is no longer the only interesting choice. As a first example of a non-regular 2ev-cover, consider the line graph of the Petersen graph.  This graph can be obtained as a 2ev-cover of $K_5$ by taking the 3-dimensional permutation representation of $S_3$, and the following assignment with $\tau_1=(1,2),\tau_2=(2,3),\tau_3=(1,3)$. \[\pmat{0&1&1&1&1\\1&0&\tau_1&\tau_2&\tau_3\\1&\tau_1&0&\tau_3&\tau_2\\1&\tau_2&\tau_3&0&\tau_1\\1&\tau_3&\tau_2&\tau_1&0}\]

\subsection{An informative non-example}
There are several ways to interpret the statement ``a cover with two new eigenvalues.'' A priori, it is not clear that we have chosen the correct interpretation. One alternative is to consider covers where the number of distinct eigenvalues has increased by 2. The following example shows that our definition is better suited to certain combinatorial investigations.

We consider a $\pm 1$ signing of $A(K_{3n})$ with three distinct eigenvalues, one of which is $-1$ (and so coincides with the spectrum of $K_{3n}$).

Construct an arc function with $f(u,v)=-1$ for edges $\{u,v\}$ of a complete bipartite subgraph $K_{n,n}$, and $f(u,v)=1$ on all other edges. This signed matrix has spectrum \[\{(2n-1)^{(2)},-1^{(3n-3)},(-n-1)^{(1)}\},\] compared to the spectrum \[\{(3n-1)^{(1)},(-1)^{(3n-2)}\}\] of $K_{3n}$. We see that the cover is possessed of two new and distinct eigenvalues and one new eigenvalue coinciding with an eigenvalue of the base graph. However it can be shown that this cover is not distance-regular.

\section{Further questions}
We have explicated the relationship between 2ev-covering graphs and the representations of their group valued adjacency matrices. Moreover we have shown that this is a natural setting in which to study combinatorial regularity of covering graphs. This study sits at the intersection of several research areas, and there are many directions in which it could progress. We name just a few that are most closely related to our present work.
\begin{itemize}
\item It may be possible to extend Theorem \ref{drg2} to start with graphs of larger diameter. However, we expect that there are no connected 2-ev covers in this case, since a result of Gardiner implies that any such cover could not be antipodal. See \cite{AG}.

\item If $Y^f$ is a 2ev-cover of $Y$ we have very strong conditions on the balancedness of  cycles in the gain graph $(Y,f)$. In Theorem \ref{drg2} the condition that $\lambda=a$ implies that for any triangle in $Y$, with vertices $a,b,c$, we have $f(a,b)f(b,c)f(c,a)=id.$ Perhaps there are cases where these conditions imply the existence of some interesting biased graphs.

\item It is possible that Theorem 5.1 can be extended to non-abelian 2ev-covers. However our proof does not seem to extend in any straightforward way. Alternatively, there may be non-abelian 2ev-covers of walk regular graphs which fail to be walk regular. We would be curious to hear of any examples. 

\item There are already a number of constructions of $\pm 1$ signings of adjacency matrices which have exactly two eigenvalues. It could be interesting to consider the associated covering graphs for some of these constructions. Moreover, it could be very interesting to extend existing methods for constructing 2ev-signings of adjacency matrices to construct cyclic 2ev-covers of larger index.

\end{itemize}

 \bibliography{2EVcites.bib}

\begin{thebibliography}{10}

\bibitem{AZ}
{\sc Alon, N., and Zheng, K.}
\newblock Unitary signings and induced subgraphs of $\mathbb{Z}_2^d$.
\newblock {\em arXiv:2003.04926\/} (2020).

\bibitem{BCKW}
{\sc Belardo, F., Cioab\u{a}, S., Koolen, J., and Wang, J.}
\newblock Open problems in the spectral theory of signed graphs.
\newblock {\em The Art Discrete and Applied Math. \textbf{1}\/} (2018), P2.10.

\bibitem{NB}
{\sc Biggs, N.}
\newblock {\em Algebraic Graph Theory, 2nd Edition}.
\newblock Cambridge Mathematical Library. Cambridge University Press, 1974.

\bibitem{BL}
{\sc Bilu, Y., and Linial, N.}
\newblock Lifts, discrepancy and nearly optimal spectral gap.
\newblock {\em Combinatorica \textbf{26}}, 5 (2006), 495--519.

\bibitem{BCN}
{\sc Brouwer, A., Cohen, A., and Neumaier, A.}
\newblock {\em Distance-Regular Graphs}.
\newblock Springer-Verlag Berlin Heidelberg, 1989.

\bibitem{CDD}
{\sc Cavaleri, M., D'angeli, D., and Donno, A.}
\newblock A group representation approach to balancedness of gain graphs.
\newblock {\em arXiv:2001.08490\/} (2020).

\bibitem{CT}
{\sc Cohen, A.~M., and Tits, J.}
\newblock On generalized hexagons and a near octagon whose lines have three
  points.
\newblock {\em Eur. J. Comb. \textbf{6}}, 1 (1985), 13--27.

\bibitem{CGSZ}
{\sc Coutinho, G., Godsil, C., Shirazi, H., and Zhan, H.}
\newblock Equiangular lines and covers of the complete graph.
\newblock {\em Linear Algebra and its Applications \textbf{488}\/} (2016), 264
  -- 283.

\bibitem{AG}
{\sc Gardiner, A.}
\newblock Antipodal covering graphs.
\newblock {\em J. Comb. Theory, Ser. {B} \textbf{16}}, 3 (1974), 255--273.

\bibitem{GF}
{\sc Ghasemian, E., and Fath-Tabar, G.}
\newblock On signed graphs with two distinct eigenvalues.
\newblock {\em Filomat \textbf{31}\/} (2017), 6393--6400.

\bibitem{GH}
{\sc Godsil, C.~D., and Hensel, A.~D.}
\newblock Distance regular covers of the complete graph.
\newblock {\em J. Comb. Theory, Ser. {B} \textbf{56}}, 2 (1992), 205--238.

\bibitem{GM}
{\sc Godsil, C.~D., and McKay, B.~D.}
\newblock Feasibility conditions for the existence of walk-regular graphs.
\newblock {\em Linear Algebra and its Applications \textbf{30}\/} (1980).

\bibitem{CG2}
{\sc Godsil, C.~D., and Royle, G.~F.}
\newblock {\em Algebraic Graph Theory}, vol.~\textbf{207} of {\em Graduate
  texts in mathematics}.
\newblock Springer-Verlag New York, 2001.

\bibitem{JG}
{\sc Gross, J.~L.}
\newblock Voltage graphs.
\newblock {\em Discrete Mathematics \textbf{9}}, 3 (1974), 239 -- 246.

\bibitem{GT}
{\sc Gross, J.~L., and Tucker, T.~W.}
\newblock Generating all graph coverings by permutation voltage assignments.
\newblock {\em Discrete Mathematics \textbf{18}}, 3 (1977), 273--283.

\bibitem{HH}
{\sc Huang, H.}
\newblock Induced subgraphs of hypercubes and a proof of the {S}ensitivity
  {C}onjecture.
\newblock {\em Annals of Mathematics \textbf{190}\/} (2019), 949--955.

\bibitem{IM}
{\sc Iverson, J.~W., and Mixon, D.~G.}
\newblock Doubly transitive lines {I:} {H}igman pairs and roux.
\newblock {\em CoRR abs/1806.09037\/} (2018).

\bibitem{JL}
{\sc James, G., and Liebeck, M.}
\newblock {\em Representations and Characters of Groups}, 2~ed.
\newblock Cambridge University Press, 2001.

\bibitem{LS}
{\sc Lemmens, P., and Seidel, J.}
\newblock Equiangular lines.
\newblock {\em Journal of Algebra \textbf{24}}, 3 (1973), 494 -- 512.

\bibitem{MSS}
{\sc Marcus, A.~W., Spielman, D.~A., and Srivastava, N.}
\newblock Interlacing families {I}: Bipartite {R}amanujan graphs of all
  degrees.
\newblock {\em Annals of Mathematics \textbf{182}}, 1 (2015), 307--325.

\bibitem{MS}
{\sc Mckee, J., and Smyth, C.}
\newblock Integer symmetric matrices having all their eigenvalues in the
  interval [-2,2].
\newblock {\em Journal of Algebra \textbf{317}\/} (2007), 260--290.

\bibitem{ZS}
{\sc Stani\'{c}, Z.}
\newblock Spectra of signed graphs with two eigenvalues.
\newblock {\em Applied Mathematics and Computation \textbf{364}\/} (2020),
  124627.

\bibitem{TZ}
{\sc Zaslavsky, T.}
\newblock Biased graphs. {I}. {B}ias, balance, and gains.
\newblock {\em J. Comb. Theory, Ser. {B} \textbf{47}}, 1 (1989), 32--52.

\bibitem{TZ2}
{\sc Zaslavsky, T.}
\newblock A mathematical bibliography of signed and gain graphs and allied
  areas.
\newblock {\em Electronic Journal of Combinatorics Dynamic Surveys No.
  \textbf{DS8} (electronic).\/} (2018).

\end{thebibliography}

\end{document}